\documentclass[11pt,twoside]{article}

\usepackage{amsthm, amsmath, amssymb, amsfonts, graphicx, epsfig, ulem}
\usepackage{accents}

\usepackage{bbm}
\usepackage{mathrsfs}
\usepackage{epstopdf}
\usepackage{graphicx}
\usepackage[font=sl,labelfont=bf]{caption}
\usepackage{subcaption}
\usepackage{float}
\restylefloat{table}
\usepackage{enumerate}
\usepackage[usenames,dvipsnames]{color}
\usepackage[colorlinks=true, pdfstartview=FitV, linkcolor=blue,
            citecolor=blue, urlcolor=blue]{hyperref}
\usepackage[usenames]{color}
\usepackage{url}
\makeatletter\def\url@leostyle{%
 \@ifundefined{selectfont}{\def\UrlFont{\sf}}{\def\UrlFont{\scriptsize\ttfamily}}} \makeatother
\usepackage[framemethod=default]{mdframed}
\usepackage[margin=1.0in, letterpaper]{geometry}

\newtheorem{theorem}{Theorem}
\newtheorem{proposition}[theorem]{Proposition}
\newtheorem{lemma}[theorem]{Lemma}
\newtheorem{corollary}[theorem]{Corollary}

\theoremstyle{definition}

\theoremstyle{remark}
\newtheorem{remark}[theorem]{Remark}

\numberwithin{equation}{section}
\numberwithin{theorem}{section}
\def\cA{\mathcal{A}}
\def\cB{\mathcal{B}}
\def\cC{\mathcal{C}}

\def\cE{\mathcal{E}}
\def\cF{\mathcal{F}}

\def\cM{\mathcal{M}}

\def\cO{\mathcal{O}}
\def\cP{\mathcal{P}}
\def\cQ{\mathcal{Q}}

\def\cT{\mathcal{T}}

\def\bE{\mathbb{E}}
\def\bF{\mathbb{F}}

\def\bN{\mathbb{N}}

\def\bP{\mathbb{P}}
\def\bQ{\mathbb{Q}}
\def\bR{\mathbb{R}}

\def\sF{\mathscr{F}}

\newcommand{\1}{\mathbbm{1}}            % preferable way of writing indicator function
\newcommand{\set}[1]{\{#1\}}            % set: {xyz} to be used for inline formulas
\renewcommand{\mid}{\;|\;}              % mid bar with small spaces before and after: x | y
\floatstyle{plaintop}

\title{ \vspace{-3em} % this negative space moves the title up
    Nonparametric Adaptive Robust Control Under Model Uncertainty
}

\author{
        Erhan Bayraktar\thanks{E. Bayraktar is partially supported by the National Science Foundation under grant DMS-2106556 and by the Susan M. Smith chair.
        \newline \hspace*{1.45em}  Department of Mathematics, University of Michigan, Ann Arbor
        \newline \hspace*{1.45em}  530 Church Street, Ann Arbor, MI 48109, USA
        \newline \hspace*{1.45em}  Email: \url{erhan@umich.edu}, URL: \url{https://sites.lsa.umich.edu/erhan/}}
        \and
        \and Tao Chen\,\thanks{
        \newline \hspace*{1.45em}  Email: \url{chenta@umich.edu}, URL: \url{http://taochen.im}
        }}

\date{ {\small %This version: \today}} %
This Version: \today
}}

\begin{document}

\maketitle

{\footnotesize
\begin{tabular}{l@{} p{350pt}}
  \hline \\[-.2em]
  \textsc{Abstract}: \ & We consider a discrete time stochastic Markovian control problem under model uncertainty. Such uncertainty not only comes from the fact that the true probability law of the underlying stochastic process is unknown, but the parametric family of probability distributions which the true law belongs to is also unknown. We propose a nonparametric adaptive robust control methodology to deal with such problem. Our approach hinges on the following building concepts: first, using the adaptive robust paradigm to incorporate online learning and uncertainty reduction into the robust control problem; second, learning the unknown probability law through the empirical distribution, and representing uncertainty reduction in terms of a sequence of Wasserstein balls around the empirical distribution; third, using Lagrangian duality to convert the optimization over Wasserstein balls to a scalar optimization problem, and adopting a machine learning technique to achieve efficient computation of the optimal control. We illustrate our methodology by considering a utility maximization problem. Numerical comparisons show that the nonparametric adaptive robust control approach is preferable to the traditional robust frameworks.\\[1em]
\textsc{Keywords:} \ & nonparametric adaptive robust control, model uncertainty, stochastic control, adaptive robust dynamic programming, Wasserstein distance, Markovian control problem, utility maximization.
 \\
\textsc{MSC2010:} \ & 49L20, 49J55, 93E20, 93E35, 60G15, 65K05, 90C39, 90C40, 91G10, 91G60, 62G05 \\[1em]
  \hline
\end{tabular}
}

\tableofcontents

\section{Introduction}
In this paper we propose a new methodology for solving a stochastic Markovian control problem in discrete time under model uncertainty.
Unlike many works in this area that assume the unknown probability law of the underlying stochastic process belongs to some parametric family of distributions, we avoid making such postulation to prevent model misspecification.
When it comes to handling model uncertainty, there are different approaches, parametric and nonparametric, developed in the past decades to incorporate learning into solving control problems with unknown system models (cf. \cite{KV2015}, \cite{CG91}, \cite{Rieder1975}, \cite{CM2020}).
However, earlier studies show that a pure learning approach without awareness of the model risk is prone to risk caused by estimation error and often leads to overly aggressive controls and system outcomes with high variances.
On the other hand, the central idea of robust control goes back to \cite{GS1989}.
A large body of research have been devoted to this area since then, and produced fruitful resutls which are briefly summarized in Section~\ref{sec:setup}.
Robust techniques are extremely successful in dealing with model risk but if the learning phase is lacking in the framework, corresponding controls can be overly conservative and even trivial.
Our work aims to address all the issues mentioned above when handling a Markovian control problem by proposing a nonparametric adaptive robust methodology and develop an efficient numerical scheme for implementing such method.

A robust control problem can be viewed as a game between the controller and the nature.
In the traditional setup, the nature chooses the worst case model against the controller at the beginning of the game. To respond, the controller adopts a control law which determines the game strategies at all time steps through the timeline.
In a sense, both counterparties' strategies are pre-committed.
Mathematically, the controller takes a set of considered models, solves the optimization problem for every model in such set, and chooses the strategy corresponding to the worst model against the controller. We refer to \cite{HSTG2006}, \cite{HS2008}, and \cite{BasarBernhardBook1995}, for more information regarding this setup.
More recent works consider a robust control problem as a sequential game: from a fixed set of models, at each time step the nature chooses one that is the worst for the controller, and the controller will apply an optimal control in response (cf. \cite{Sirbu2014}, \cite{BCP2016}).
The main difference between the two approaches mentioned so far is that the worst case model is time independent in the former case and time dependent in the latter.
In \cite{Nutz2016}, the author presented a robust framework where the nature chooses models from a time dependent set.
In other words, the nature can pick strategies from different sets of available actions at different stages of the game.
In \cite{BCCCJ2019}, the authors specified the dynamics of such sets via recursive confidence regions of the unknown model parameters. We refer to \cite{BCC2017} for the detailed discussion of recursive construction of confidence regions.
Such idea is also utilized in this work.
The advantage of using confidence regions are twofold.
On one hand, as new realization of the random noise in the system is observed between the decision-making time points, the confidence region updates itself and naturally represents the learning of the unknown system model.
On the other hand, such sequence of sets is asymptotically shrinking in size which leads to reduction of the model uncertainty.
To the best of our knowledge, \cite{BCCCJ2019} is the first work that incorporates the idea of online learning into the robust control paradigm.
A follow-up work in \cite{BC2021} is an attempt to extend the adaptive robust control to the continuous time setup.

Note that the methods in \cite{BCCCJ2019} and \cite{BC2021} are parametric and the practical usage of such methods relies on the assumption that the family of the unknown probability law of the underlying stochastic process is known to the controller.
Some researchers have realized this drawback and adopts nonparametric statistical methods by assuming uncertainty for the family of parametric models.
To formulate a robust setup, one will define a set of probability distributions that includes the estimated distribution.
For example, in \cite{KENA2019} and \cite{OW2021}, the authors take a Wasserstein ball around the empirical distribution and use the ball as the set of considered models.
However, such setup has only been implemented in one-period control problems so far, and the feasibility of this approach in multi-period setup remains to be investigated.
To overcome this obstacle, we develop a nonparametric adaptive robust control methodology in this work to handle multi-period stochastic control problems where the family of distributions which the true law of the system model belongs to is unknown.
Naturally, we use the empirical distribution as the estimate of the distribution of the underlying stochastic process.
Another candidate for this purpose is the perturbed empirical distribution when such distribution is known to be continuous.
For construction of confidence regions in this setup, we utilize the Wasserstein ball around the empirical distribution.
There are several works on the concentration results regarding the empirical distribution and the Wasserstein distance (cf. \cite{BGM1999}, \cite{FG2015}).
Backed by these papers, one obtains a CLT-type of result for the empirical distribution that leads to construction of confidence regions through Wasserstein distance under rather mild assumptions.
Practically, numerical search of the worst case model in a set of probability distributions is extremely difficult.
Another advantage of using the Wasserstein ball as the confidecnce region is that the aforementioned task of searching for the worst case model in a Wasserstein ball can be converted to a scalar optimization problem.
Last but not the least, we implement a machine learning technique via the Gaussian process surrogates \cite{RW2006} to build regression models for the relevent value function and the optimal control.
The former surrogate enables us to proceed the backward recursion according to the dynamic programming principle, and the latter allows us for fast computation of the optimal control when applying our framework.

The rest of the paper is organized as follows. We begin Section~\ref{sec:setup} with setting up the model and in Section~\ref{sec:uncertainty_set} we discuss the contruction of confidence region for the unknown true probability law in terms of the Wasserstein ball. Such sets of distributions represent the uncertainty of the system model.
Section~\ref{sec:np_ar} is dedicated to the formulation of the nonparametric adaptive robust control framework.
We investigate the solution of the nonparametric adaptive robust control problem and derive the associated Bellman equations in Section~\ref{sec:solution}.
Also in this section, we prove the Bellman principle of optimality for the problem and show the existence of measurable worst-case model selector as well as the existence of measurable optimal control.
In Section~\ref{sec:convergence}, we discuss the convergence and deviation of the adaptive robust value function to the true value function.
Finally, in Section~\ref{sec:numerical} we consider an illustrative example.
Namely, the uncertain utility maximization problem where the investor needs to allocate the wealth between the money market account and the risky asset without knowing the true distribution of the risky asset's return process.
We apply the nonparametric adaptive robust control approach to such problem and provide a numerical solver by using machine learning techniques.
Numerical results presented in this section show the favorable aspects of the proposed methodology to the traditional robust control framework and the case of knowing the true model.

\section{Nonparametric Stochastic Control Problem Subject to Model Uncertainty}\label{sec:setup}

Let  $(\Omega, \sF)$ be a measurable space,  and $T\in \bN$ be a fixed time horizon. Let $\cT=\set{0,1,2,\ldots,T}$, $\cT'=\set{0,1,2,\ldots,T-1}$, and $\cT''=\set{1,2,\ldots,T}$.
On the space $(\Omega, \sF)$ we consider a controlled random process $X=\{X_t,\ t\in\cT\}$ taking values in $\bR^n$ with dynamics
\begin{align}
  X_{t+1}=S(X_t,\varphi_t,Z_{t+1}), \quad t\in\cT', \quad X_0=x_0\in\bR^n.
\end{align}
The above $Z=\{Z_t, t\in\cT\}$ is an i.i.d. real valued random sequence of which the natural filtration is denoted by $\bF=(\cF_t, t\in\cT)$. The process $\varphi=\set{\varphi_t, t\in\cT'}$ is $\bF$-adapted and takes values in a compact set $A$. The function $S:\bR^n\times A\times\bR\to\bR^n$ is deterministic and continuous. For every $t\in\cT'$, we denote by $\cA_t$ the set of all processes that take values in $A$ and are adapted to the filtration $\bF_t:=(\cF_s, t\leq s\leq T-1)$. Each element in $\cA_t$ is called an admissible control starting at time $t$, and we use the convention $\cA=\cA_0$. In this work, we assume that the process $Z$ is observable but the distribution $F^*$ of each $Z_t$ is unknown. We write $\cP(\bR)$ as the set of all distributions on $\bR$ and $\bP_F$ as the probability measure on $(\Omega,\sF)$ corresponds to $F\in\cP(\bR)$. The expectation associated to $\bP_F$ is  $\bE_F$, and $\ell:\bR^n\to\bR$ is the loss function which is continuous. In this work we will formulate and solve a robust optimization problem aiming to minimize the expected loss when taking into consideration that the true distribution $F^*$ of $Z$ is unknown.
In order to avoid model misspecification caused by assuming a wrong parametric family of distrubtions, we will conduct online learning of the underlying system in a nonparametric manner via empirical distribution. In the spirit of \cite{BCCCJ2019}, we define the sets of model candidates as approximated confidence regions around the empirical distribution. Such sets are Wasserstein balls and their sizes decreases as time goes on in general. Therefore, in our robust framework, the uncertainty is dynamically reduced through online learning and shrinkage of the Wasserstein balls.

\subsection{Empirical Distribution and Uncertainty Set}\label{sec:uncertainty_set}

We make a standing postulation that $F^*$ satisfies that
\begin{align}\label{eq:L21}
\int_{-\infty}^{\infty}\sqrt{F^*(z)(1-F^*(z))}dz<\infty.
\end{align}
We note that any distribution that has finite moments with order higher than 2 will satisfy the above assumption.
Next, denote by $\widehat{F}_t$, $t\in\cT'$, the empirical distrubtion of $Z_{t+1}$ given the initial guess $\widehat{F}_0$ of $F^*$ and the observations $Z_{1:t}:=\set{Z_i, i=1,\ldots,t}$, where $\widehat{F}_0$ is the empirical distribution of $Z_1$ based on historical data of $Z$ with sample size $t_0$. In other words, $\widehat{F}_t$ is the contructed based $Z_{-t_0+1:t}$. Defined as an average of indicator functions, $\widehat{F}_t$ satisfies the following recursion similarly to any estimated mean:
\begin{align}\label{eq:f_map}
\widehat{F}_{t+1}(z)=\frac{(t_0+t)\widehat{F}_t+\1_{\{Z_{t+1}<z\}}}{t_0+t+1}:=R(t,\widehat{F}_t,Z_{t+1}), \quad z\in\bR,\ t\in\cT'.
\end{align}
The map $R$ defined above will be viewed as the dynamics of the process $\widehat{F}$. Regarding other properties of $\widehat{F}$, 
it is well known that $\widehat{F}$ is a consistent estimator of $F^*$:
$$
\lim_{t\to\infty}\widehat{F}_t(z)=F^*(z), \quad \text{a.s..}
$$
Moreover, by the assumption \eqref{eq:L21} and using the results in \cite{BGM1999}, we have that
\begin{align}\label{eq:clt}
\sqrt{t_0+t}d_{W,1}(\widehat{F}_t,F^*)=\sqrt{t_0+t}\int_{-\infty}^\infty|\widehat{F}_t(z)-F^*(z)|dz\to\int_0^1|B(s)|dQ^*(s),
\end{align}
where $d_{W,1}$ is the Wasserstein distance of order 1, $B(s)$, $0\leq s\leq1$, and $Q^*$ are the Brownian bridge and the quantile function of $F^*$, respectively, and the convergence is in distribution. We will construct an approximated confidence region for $F^*$ based on \eqref{eq:clt}. Since $Q^*$ is unknown, we will approximate it by using $\widehat{Q}_t$ which is the quantile function corresponding to $\widehat{F}_t$. At time $t$, the integral $\int_0^1|B(s)|dQ^*(s)$ is then approximated as
\begin{align*}
\int_0^1|B(s)|dQ^*(s)\approx\sum_{i=1}^{t_0+t-1}\left|B\left(\frac{i}{t_0+t}\right)\right|\left(z_{(i+1)}-z_{(i)}\right)=:H_t(\widehat{F}_t),
\end{align*}
where $z_{(1):(t_0+t)}$ is the order statistics of $z_{-t_0+1:t}$. We define the $\alpha$-uncertainty set $\cC^{\alpha}_t$, $0<\alpha<1$, which is an approximated confidence region for $F^*$ as
\begin{align}\label{eq:uncertainty_set}
\cC^{\alpha}_t(\widehat{F}_t)=\left\{F\in\cP_1(\bR): d_{W,1}(\widehat{F}_t,F)\leq\frac{Q^H_t(1-\alpha)}{\sqrt{t_0+t}}\right\},
\end{align}
where $\cP_1(\bR)$ is the set of all distributions with finite first moment, and $Q^H_t$ is the quantile function of $H_t(\widehat{F}_t)$. Due to the discussion above, the rational behind \eqref{eq:uncertainty_set} is that the probability that $C^{\alpha}_t(\widehat{F}_t)$ contains $F^*$ is approximately $1-\alpha$. 
Note that theoretically we can derive the distribution of $\sum_{i=1}^{t_0+t-1}|B(\frac{i}{t_0+t})|(z_{(i+1)}-z_{(i)})$.
But since $B(\frac{i}{t_0+t})$, $i=1,\ldots,t_0+t-1$, are not independent, then such computation will be too tedious. 
Hence, we will estimate $Q^H_t(1-\alpha)$ via simulation instead.
As another way to justify that the radius of the Wasstertein ball being a multiple of $\frac{1}{\sqrt{t_0+t}}$ is the calculation from \cite{FG2015} where they show, under the stronger assumption
\begin{align}\label{eq:fe_condition}
\int_{\bR}e^{c_1|z|^{c_2}}F^*(dz)<\infty,
\end{align}
for some $c_1>0$, and $c_2>1$, that for any fixed $0<\alpha<1$, there exists some constants $C$ and $c$ such that 
\begin{align}\label{eq:concentration}
\bP\left(d_{W,1}(\widehat{F}_t,F^*)\leq\sqrt{\frac{\log(C/\alpha)}{c(t+t_0)}}\right)\geq 1-\alpha.
\end{align}
A different formulation of the uncertainty sets can be obtained from $\eqref{eq:concentration}$.
However, radius of the resulting Wasserstein ball has the same order, namely, $\frac{1}{\sqrt{t_0+t}}$ as of $\cC^{\alpha}_t$.

Next, by using a different representation of the Wasserstein distance between probability distributions, we have the following technical result for the map defined in \eqref{eq:f_map}.

\begin{lemma}\label{lemma:f_map_cont}
For fixed $t\in\cT'$, the mapping $R(t,\cdot,\cdot):\cP_1(\bR)\times\bR\to\cP_1(\bR)$ is continuous.
\end{lemma}

\begin{proof}
Assume that $(F_n,z_n)\to(F,z)$ where $F_n,\ F\in\cP_1(\bR)$, $z_n,\ z\in\bR$, $n=1,\ 2,\ ,\ldots$. Then, $d_{W,1}(F_n,F)\to0$ and $z_n\to z$. Denote $\mu_{F_n,z_n}=R(t,F_n,z_n)$ and $\mu_{F,z}=R(t,F,z)$. For $\cM:=\set{f:|f(x)-f(y)|\leq|x-y|}$, we have that
\begin{align*}
d_{W,1}(\mu_{F_n,z_n},\mu_{F,z})&=\sup\left\{\int_{\bR}fd\mu_{F_n,z_n}-\int_{\bR}fd\mu_{F,z}:\ f\in\cM\right\}\\
&=\sup\left\{\frac{t_0+t}{t_0+t+1}\left(\int_{\bR}fdF_n-\int_{\bR}fdF\right)+\frac{1}{t_0+t+1}(f(z_n)-f(z)):\ f\in\cM\right\}\\
&\leq \frac{t_0+t}{t_0+t+1}\sup\left\{\int_{\bR}fdF_n-\int_{\bR}fdF:\ f\in\cM\right\}+\frac{1}{t_0+t+1}|z_n-z|\\
&=\frac{t_0+t}{t_0+t+1}d_{W,1}(F_n,F)+\frac{1}{t_0+t+1}|z_n-z|.
\end{align*}
Therefore, we get that $d_{W,1}(\mu_{F_n,z_n},\mu_{F,z})\to0$ and the mapping $R(t,\cdot,\cdot)$ is continuous.
\end{proof}

One property that the set valued function $\cC^{\alpha}_t$ satisfies is upper hemicontinuity (u.h.c.). That is for any for any $F\in\cP_1(\bR)$ and any open set $E$ such that $\cC^{\alpha}_t(F)\subset E\subset\cP_1(\bR)$, there exists a neighbourhood $D$ of $F$ such that for all $F'\in D$, $\cC^{\alpha}_t(F')\subset E$ (cf. \cite[Definition 11.3]{Border1985}). To see that $\cC^{\alpha}_t$ is u.h.c., let $\varepsilon=\text{dist}(F,\partial E)-Q^H_t(1-\alpha)/\sqrt{t_0+t}$ where $\text{dist}(F,\partial E)$ is the shortest distance from $F$ to the boundary of $E$. Then, take $D$ as the ball centered at $F$ with radius $\varepsilon$. It is not hard to see that for any $F'\in D$, we have $\cC^{\alpha}_t(F')\subset E$. We summarize the result as follows
\begin{lemma}\label{lemma:uhc}
For every $t\in\cT'$, the set valued function $\cC^{\alpha}_t$ is upper hemicontinuous.
\end{lemma}
As per our discussion above, the proof is straightforward and we omit it here.

\subsection{Nonparametric Adaptive Robust Control Problem}\label{sec:np_ar}

Now we proceed to formulate the nonparametric adaptive robust control problem. For the rest of the paper, we will consider $\cP_1(\bR)$ with the metric $d_{W,1}$. Since $\bR$ is separable and complete, then $(\cP_1(\bR),d_{W,1})$ is also separable and complete. Hence, $\cP_1(\bR)$ is a Polish space and thus a Borel space. Define the augumented state process $Y=\{Y_t=(X_t,\widehat{F}_t)$, $t\in\cT\}$, and the augumented state space $E_Y=\bR^n\times\cP_1(\bR)$. For $E_Y$ we equip the product topology, it is then a Borel space and the Borel $\sigma$-algebra $\cE_Y$ conicides with the product $\sigma$-algebra. The process $Y$ has the following dynamics
\begin{align}
Y_{t+1}=\mathbf{G}(t,Y_t,\varphi_t,Z_{t+1}):=(S(X_t,\varphi_t,Z_{t+1}),R(t,\widehat{F}_t,Z_{t+1})), \quad t\in\cT'.
\end{align}

According to the assumption that $S$ is continuous and Lemma~\ref{lemma:f_map_cont}, we get that $\mathbf{G}$ is continuous and therefore Borel measurable. Next, given our setup, the process $Y$ is $\bF$-adapted and Markovian.
The transition probability for the state process $Y$ is defined as follows. For any $t\in\cT'$, $(y,a)\in E_Y\times A$, and $F\in\cP_1(\bR)$, $Q_t$ is a probability measure on $\cE_Y$ such that
\begin{align*}
Q_t(D|y,a,F)=\bP_{F}(\mathbf{G}(t,y,a,Z_{t+1})\in D), \quad D\in\cE_Y.
\end{align*}

One important property of the stochastic kernel $Q_t$ is that it is in fact Borel measurable which will be proved below. Such property is crucial for showing the existence of measurable optimal controls.

\begin{proposition}\label{prop:borel}
For each $t\in\cT'$, the probability $Q_t(\ \cdot\ |y,a,F)$ is a Borel measurable stochastic kernel on $E_Y$ given $E_Y\times A\times\cP_1(\bR)$.
\end{proposition}

\begin{proof}
According to \cite{BS1978}, it is enough to show that for any $b_1, b_2\in\bR^n$ and closed ball $D\subseteq\cP(\bR)$ with finite radius, $Q_t([b_1,b_2]\times D|y,a,F)$ is a measurable function in $(y,a,F)$, where
$$
[b_1,b_2]:=\textsf{X}_{i=1}^n[b_1^{(i)},b_2^{(i)}], \quad b_j=(b_j^{(1)},\ldots,b_j^{(n)}),\ j=1,2.
$$
We will prove that $Q_t([b_1,b_2]\times D|y,a,F)$ is upper semi-continuous, and then it will be Borel measurable.

Fix any $(y_0,a_0,F_0)\in E_Y\times A\times\cP_1(\bR)$, and let $\{(y_n,a_n,F_n),n>0\}$ be a sequence that converges to $(y_0,a_0,F_0)$. Note that the set $C_0:=\{z:\mathbf{G}(t,y_0,a_0,z)\in [b_1,b_2]\times D\}$ is a closed set since the map $\mathbf{G}$ is continuous. We similary define $C_n=\{z:\mathbf{G}(t,y_n,a_n,z)\in[b_1,b_2]\times D\}$, $n>0$, and they satisfy the same properties.

We first prove that $\bigcup_{i=0}^\infty C_i$ is bounded. Assume the union contains at least two points. If the two points $z_1<z_2$ belong to the same $C_n$, denote by $\hat{f}_n$ the second component of $y_n$, we have that
\begin{align*}
d_{W,1}(\mu_{\hat{f}_n,z_1},\mu_{\hat{f}_n,z_2})&=\sup\left\{\int_{\bR}gd\mu_{\hat{f}_n,z_1}-\int_{\bR}gd\mu_{\hat{f}_n,z_2}:g\in\cM\right\}\\
&=\sup\left\{\frac{g(z_1)-g(z_2)}{t_0+t+1}:g\in\cM\right\}\\
&\geq \frac{z_2-z_1}{t_0+t+1}.
\end{align*}
Since $D$ is a bounded set, then $z_2$ must be within a bounded range of $z_1$. Next assume that there are $z_k\in C_k$ and $z_l\in C_l$, and $z_l>z_k$. Again, we have
\begin{align*}
d_{W,1}(\mu_{\hat{f}_l,z_l},\mu_{\hat{f}_k,z_k})&=\sup\left\{\int_{\bR}gd\mu_{\hat{f}_l,z_l}-\int_{\bR}gd\mu_{\hat{f}_k,z_k}:g\in\cM\right\}\\
&\geq\frac{t_0+t}{t_0+t+1}(\bE_{\hat{f}_l}[Z_l]-\bE_{\hat{f}_k}[Z_k])+\frac{z_l-z_k}{t_0+t+1}.
\end{align*}
Since $\{y_n,n>0\}$ is a convergent sequence, then the value of the first term on the right hand side of the above inequality is bounded for any $k$ and $l$.
Moreover, the value $z_l-z_k$ should be bounded as well. Now we see that $\bigcup_{i=0}^\infty C_i$ is bounded and every single $C_n$ is compact.
Next, we show that if $C_0=\emptyset$ then for large enough $n$ the set $C_n$ is also empty.
In particular, if the preimage $R^{-1}(t,\hat{f}_0,D)=\emptyset$, then for large enough $n$, $R^{-1}(t,\hat{f}_n,D)=\emptyset$.
Otherwise, we can find a subsequence $n_k$, $k>0$, such that there exist $z_{n_k}\in R^{-1}(t,\hat{f}_0,D)$ for all $k$. Without loss of generality, we assume that $z_{n_k}$ is convergent due to the fact that $\bigcup_{i=0}^\infty C_i$ is compact. Then, $R(t,\hat{f}_{n_k},z_{n_k})$, $k>0$, is a convergent sequence and $R(t,\hat{f}_{n_k},z_{n_k})\in D$, $k>0$. Because $D$ is closed, the limit $R(t,\hat{f}_0,z_0)\in D$ which implies that $z_0\in C_1$. This contradicts to the assumption that $R^{-1}(t,\hat{f}_0,D)=\emptyset$, hence for large enough $n$, $R^{-1}(t,\hat{f}_n,D)=\emptyset$. On the other hand, by using the continuity argument, it is also easy to see that if $S(x_0,a_0,z)\notin[b_1,b_2]$ for all $z\in\bR$, then for large enough $n$, $S(x_n,a_n,z)\notin[b_1,b_2]$ for all $z\in\bR$. We have proved that if $C_0=\emptyset$ then for large enough $n$ the set $C_n$ is also empty. In this case,
$$
\lim_{n\to\infty}Q_t([b_1,b_2]\times D|y_n,a_n,F_n)=Q([b_1,b_2]\times D|t,y_0,a_0,F_0)=0,
$$
and the function $Q_t([b_1,b_2]\times D|y,a,F)$ is continuous and therefore upper semi-continuous at such $(y_0,a_0,F_0)$.

For the rest of the proof, we assume that $C_0\neq\emptyset$.
Let $\varepsilon_m>0$, $m>0$, be a strictly decreasing sequence that converges to 0. 
For any $\varepsilon_m$ and $z\in\bR$, let $\cB_m(z)$ be the open ball centered at $z$ with radius $\varepsilon_m$. The collection $\{\cB_m(z):z\in C_0\}$ is an open cover of the compact $C_0$, and there exists a finite subcover $\cB_m(z_{(1)}),\ldots,\cB_m(z_{(k_m)})$. Define the set $C_0^m=\bigcup_{i=1}^{k_m}\overline{\cB_m(z_{(i)})}$, and we argue that for any $m>0$, there exists $N_m>0$ such that for any $n>N_m$, we have $C_n\subseteq C^m_0$.

We prove the above statement by contradiction. Assume that it is not true. Then for any $N>0$, there exits $n>N$ such that $C_n\not\subseteq C^m_0$. Consequently, there exists a sub-sequence $n_j$, $j>0$, such that $z_{n_j}\in C_{n_j}$ but $z_{n_j}\notin C^m_0$.
From previous discussions we know the sequence $z_{n_j}$ is bounded, and moreover there exists a $z^*$ that is a limiting point of $z_{n_j}$. It is safe to assume
\begin{align}\label{eq:zstarnotinC}
z^*\notin C^m_0
\end{align}
for the reason that if $z^*$ is on the boundary of $C^m_0$, we can replace $\varepsilon_m$ with a number in the interval $(\varepsilon_{m+1},\varepsilon_m)$. Let us consider the sequence $\{\mathbf{G}(t,y_{n_j},a_{n_j},z_{n_j}), j>0\}$. Recall that $z_{n_j}\in C_{n_j}$, hence $\mathbf{G}(t,y_{n_j},a_{n_j},z_{n_j})\in[b_1,b_2]\times D$ for all $j>0$. Due to Lemma~\ref{lemma:f_map_cont}, $\mathbf{G}(t,y_0,a_0,z^*)$ is a limiting point of such sequence.
In addition, since $[b_1,b_2]\times D$ is a closed set, then $z^*\in C_0\subseteq C^m_0$ which contradicts to \eqref{eq:zstarnotinC}. Now we conclude that any $m>0$, there exists $N_m>0$ such that for any $n>N_m$, we have $C_n\subseteq C^m_0$.

Next, we obtain that
\begin{align}\label{eq:Qn}
Q_t([b_1,b_2]\times D|y_n,a_n,F_n)=\bP_{F_n}(C_n)\leq \bP_{F_n}(C^m_0).
\end{align}
Since $C^m_0$ is a closed set, and $F_n$ converges weakly to $F_0$, then \eqref{eq:Qn} implies that
\begin{align*}
\limsup_nQ_t([b_1,b_2]\times D|y_n,a_n,F_n)=\limsup_n\bP_{F_n}(C_n)\leq \limsup_n\bP_{F_n}(C^m_0)\leq \bP_{F_0}(C^m_0).
\end{align*}
Finally, note that one can construct $\{C^m_0,m>0\}$ such that the sequence of sets is decreasing and $\bigcap_mC^m_0=C_0$. We have
$$
\lim_{m\to\infty}\bP_{F_0}(C^m_0)=\bP_{F_0}(C_0).
$$
It follows immediately that
$$
\limsup_nQ_t(b_1,b_2]\times D|y_n,a_n,F_n)\leq Q_t([b_1,b_2]\times D|y_0,a_0,F_0).
$$
To summarize, we obtain that $Q_t([b_1,b_2]\times D|\ \cdot,\ \cdot,\ \cdot\ )$ is upper semi-continuous. Therefore, it is a Borel measurable function.
\end{proof}

In this work, we are dealing with a closed loop feedback control problem. 
To this end, a control process $\varphi$ is called Markovian if for every $t\in\cT'$ (with a slight abuse of notation)
$$
\varphi_t=\varphi_t(Y_t)
$$
where on the right hand side $\varphi_t:E_Y\to A$ is a measurable mapping. Similarly, A process $\psi$ is called a Markovian model selector if
$$
\psi_t=\psi_t(Y_t)
$$
where $\psi_t:E_Y\to\cP_1(\bR)$ is measurable. In the adaptive robust framework, we consider the Markovian control processes and Markovian model selectors such that $\psi_t(y)\in\cC^{\alpha}_t(y)$ for any $y\in E_Y$.
For every $t\in\cT'$, any time $t$ state $y_t\in E_Y$, and control process $\varphi\in\cA_t$, we denote
$$
\mathbf{\Psi}^{\varphi}_{y_t,t}=\left\{\psi_{t:T-1}, \psi_s(y_s)\in\cC^{\alpha}_s(y_s), \exists z\in\bR, \text{s.t. } y_{s+1}=\mathbf{G}(s,y_s,\varphi_s(y_s),z), t\leq s<T-1\right\}.
$$
and
$$
\mathbf{\Psi}_{y_t,t}=\left\{\psi_{t:T-1}, \psi_s(y_s)\in\cC^{\alpha}_s(y_s), \exists a\in A, z\in\bR, \text{s.t. } y_{s+1}=\mathbf{G}(s,y_s,a,z), t\leq s<T-1\right\}.
$$
Next, for every $t\in\cT'$, any $y_t\in E_Y$, $\varphi\in\cA_t$, and $\psi\in\mathbf{\Psi}_{y_t,t}$, we define the probability measure $\bQ^{\varphi,\psi}_{y_t,t}$ on the concatenated canonical space $\textsf{X}_{s=t+1}^TE_Y$ as
\begin{align*}
\bQ^{\varphi,\psi}_{y_t,t}(B_{t+1}\times\cdots\times B_T)=\int_{B_{t+1}}\cdots\int_{B_T}\prod_{u=t}^{T-1}Q_u(dy_{u+1}|y_u,\varphi_u(y_u),\psi_u(y_u)).
\end{align*}
Correspondingly, we define the family of probability measures $\cQ^\varphi_{y_t,t}=\{\bQ^{\varphi,\psi}_{y_t,t},\psi\in\mathbf{\Psi}_{y_t,t}\}$. In particular, we let $\cQ^{\varphi}_{y_0}=\cQ^{\varphi}_{y_0,0}$. Then, for given $y_0\in E_Y$, the nonparametric adaptive robust control problem is formulated as
\begin{align}\label{eq:nar}
\inf_{\varphi\in\cA}\sup_{\bQ\in\cQ^{\varphi}_{y_0}}\bE_{\bQ}[\ell(X_T)].
\end{align}
In a traditional robust setup, one would choose a fixed set $\cP_0$ in place of $\cC^\alpha_t$. Due to such reason, we will call it the static robust framework throughout. In comparison, the advantage of \eqref{eq:nar} is that such framework integrates robust control with learning and reducing uncertainty. The learning of the unknown model is carried through via the evolution of the process $Y$, and reduction of uncertainty is embedded in the construction of $\cQ^{\varphi}_{y_0}$ since for any $y_t\in E_Y$ instead of finding the worst case model in the fixed set $\cP_0$, the selectors take values in the uncertainty sets $\cC^\alpha_t(y_t)$ which is a sequence of random sets that shrink in size.

\subsection{Solution of Nonparametric Adaptive Robust Control Problem}\label{sec:solution}

We will show that solution of the nonparametric adaptive robust control problem is given by solving the following adaptive robust Bellman equations
\begin{align}\label{eq:bellman}
V_T(y) &= \ell(x), \quad y\in E_Y,\nonumber \\
V_t(y) &= \inf_{a\in A}\sup_{F\in\cC^\alpha_t(y)}\int_{E_Y}V_{t+1}(y_{t+1})Q_t(dy_{t+1}|y,a,F), \quad y\in E_Y, t\in\cT'.
\end{align}

Before we prove the main theorem in this section, let us first provide the following technical result.

\begin{lemma}\label{lemma:open}
Fix $t\in\cT'$, for any $\widehat{F}\in\cP_1(\bR)$, let
$$
\widetilde{\cC}^{\alpha}_t(\widehat{F})=\left\{F\in\cP_1(\bR):d_{W,1}(F,\widehat{F})<\frac{Q^H_t(1-\alpha)}{\sqrt{t_0+t}}\right\}.
$$
Then,
$$
\cO^\alpha_t:=\bigcup_{y\in E_Y}\widetilde{\cC}^{\alpha}_t(y)
$$
is an open set in $E_Y\times\cP_1(\bR)$.
\end{lemma}

\begin{proof}
We prove the statement by contradiction. Assume there exists $(y_0,F_0)\in\cO^{\alpha}_t$, and there exists a sequence $(y_n,F_n)\to(y_0,F_0)$, such that for any $n>0$, $(y_n,F_n)\notin\cO^{\alpha}_t$.
Note that $F_0\in\widetilde{\cC}^{\alpha}_t(y_0)$, hence $d_{W,1}(F_0,\hat{f}_0)<Q^H_t(1-\alpha)/\sqrt{t_0+t}-\varepsilon$ for some $\varepsilon>0$, where $\hat{f}_0$ is the second component of $y_0$. We have
\begin{align*}
d_{W,1}(F_n,\hat{f}_n)\leq d_{W,1}(F_n,F_0)+d_{W,1}(F_0,\hat{f}_0)+d_{W,1}(\hat{f}_0,\hat{f}_n).
\end{align*}
For large enough $n$, we have $d_{W,1}(F_n,F_0)<\varepsilon/4$, and $d_{W,1}(\hat{f}_0,\hat{f}_n)<\varepsilon/4$. Then, for such $n$, the following equality holds true
$$
d_{W,1}(F_n,\hat{f}_n)\leq\frac{Q^H_t(1-\alpha)}{\sqrt{t_0+t}}-\frac{\varepsilon}{2},
$$
which implies that $F_n\in\widetilde{\cC}^{\alpha}_t(y_n)$ and $(y_n,F_n)\in\cO^{\alpha}_t$. We get the contradiction so the set $\cO^{\alpha}_t$ is open.
\end{proof}

Next, we have the main result of this section which shows that the optimal control $\varphi$ and model selector $\psi$ exist, and they are sequences of measurable functions.

\begin{theorem}\label{thm:selector}
For every $t\in\cT$, the function $V_t$ is lower semicontinuous (l.s.c.) and upper semianalytic (u.s.a.). There exists Borel measurable optimal control $\varphi^*_t$, $t\in\cT'$, and universally measurable model selector $\psi^*_t$, $t\in\cT'$.
\end{theorem}

\begin{proof}
Since $V_T(y)=\ell(x)$ which is continuous by assumption, then $V_T$ is l.s.c. and u.s.a.. Next, denote
\begin{align*}
v_{T-1}(y,a,F)=\int_{E_Y}V_T(y_T)Q_{T-1}(dy_T|y,a,F).
\end{align*}
By using Proposition~\ref{prop:borel}, we have that $v_{T-1}(y,a,F)$ is u.s.a.. Let $D=\bigcup_{(y,a)\in E_Y\times A}\cC^{\alpha}_{T-1}(y)$. Note that $\cC^{\alpha}_{T-1}$ is u.h.c. from Lemma~\ref{lemma:uhc} and closed valued, by adopting the proof of \cite{BCC2021} in our setup, we obtain that the graph of $\cC^{\alpha}_{t-1}$, which is $D$, is closed. Hence, the set $D$ is analytic. The $(y,a)$ section of $D$ is $\cC^{\alpha}_{T-1}(y)$, and according to \cite{BS1978}, we get that
$$
\widetilde{v}_{T-1}(y,a):=\sup_{F\in\cC^{\alpha}_{T-1}(y)}v_{T-1}(y,a,F)
$$
is u.s.c.. Moreover, for any $\varepsilon>0$, there exists an analytically measurable function $\widetilde{\psi}:E_Y\times A$ such that for any $(y,a)$,
\begin{align}\label{eq:approx_selector}
v_{T-1}(y,a,\widetilde{\psi}(y,a))\geq
\begin{cases}
\widetilde{v}_{T-1}(y,a)-\varepsilon, \quad &\text{if } \widetilde{v}_{T-1}(y,a)<\infty,\\
1/\varepsilon, \quad &\text{if } \widetilde{v}_{T-1}(y,a)=\infty.
\end{cases}
\end{align}
Define the set
$$
I = \left\{(y,a)\in E_Y\times A: \text{for some } F^*\in\cC^{\alpha}_{T-1}(y), v_{T-1}(y,a,F^*)=\widetilde{v}_{T-1}(y,a)\right\},
$$
and we claim that $I=E_Y\times A$. That is for any $(y,a)\in E_Y\times A$ there exists an $F^*$ such that $v_{T-1}(y,a,F^*)=\widetilde{v}_{T-1}(y,a)$. To see why this is true, by taking in \eqref{eq:approx_selector} $\varepsilon=1/n$, we obtain a sequence $\widetilde{\psi}_n$ such that
$$
\lim_{n\to\infty}v_{T-1}(y,a,\widetilde{\psi}_n(y,a))=\widetilde{v}_{T-1}(y,a).
$$
Next, note that $\cC^{\alpha}_{T-1}(y)$ is weakly compact, so there exists $\widetilde{\psi}^*(y,a)$ as a limiting point of $\widetilde{\psi}_n(y,a)$, $n>0$, such that $v_{T-1}(y,a,\widetilde{\psi}^*(y,a))=\widetilde{v}_{T-1}(y,a)$, and indeed $I=E_Y\times A$. Therefore, by \cite{BS1978}, there exists a universally measurable function $\psi^*_{T-1}:E_Y\times A\to\cC^{\alpha}_{T-1}(y)$ which satisfies
$$
v_{T-1}(y,a,\psi^*_{T-1}(y,a))=\widetilde{v}_{T-1}(y,a).
$$

Now we prove that the function $\widetilde{v}_{T-1}(y,a)$ is l.s.c.. To this end, we write
$$
v_{T-1}(y,a,F)=\int_{\bR}V_T(\mathbf{G}(T-1,y,a,z))dF(z).
$$
Since $V_T$ is l.s.c. and $\mathbf{G}(T-1,y,a,z)$ is continuous in $(y,a,z)$, then $V_T(\mathbf{G}(T-1,y,a,z))$ is l.s.c.. On the other hand, $F$ is clearly a continuous stochastic kernel on $\bR$ given $\cP_1(\bR)$. In view of the assumption that $\ell$ is bounded below, we know the function $v_{T-1}(y,a,F)$ is l.s.c.. Let us consider the optimization problem
$$
\widehat{v}_{T-1}(y,a)=\sup_{F\in\widetilde{\cC}^{\alpha}_t(y)}v_{T-1}(y,a,F).
$$
Lemma~\ref{lemma:open} shows that the set $\cO^{\alpha}_t$ is open in $E_Y\times\cP_1(\bR)$, so it is also open in $D':=E_Y\times A\times\cP_1(\bR)$. The $\widetilde{\cC}^{\alpha}_t(y)$ is the $(y,a)$ section of $D'$. By \cite{BS1978}, we obtain that $\widehat{v}_{T-1}(y,a)$ is l.s.c.. Note for any $y\in E_Y$, the uncertainty set $\cC^{\alpha}_t(y)$
is the closure of $\widetilde{\cC}^{\alpha}_t(y)$, it follows immediately that $\widetilde{v}_{T-1}(y,a)=\widehat{v}_{T-1}(y,a)$ and the former is therefore l.s.c..

It remains to show that
\begin{align*}
V_{T-1}(y)=\inf_{a\in A}\widetilde{v}_{T-1}(y,a)
\end{align*}
is l.s.c., and there exists a Borel measruable function $\varphi^*:E_Y\to A$ such that
\begin{align*}
V_{T-1}(y)=\widetilde{v}_{T-1}(y,\varphi^*(y)).
\end{align*}
Towards this end, we note that $D''=E_Y\times A$ is closed, and $A$ by assumption is compact. The $y$ section of $D''$ is $A$ for any $y\in E_Y$. Thus, by \cite{BS1978}, the function $V_{T-1}$ is l.s.c., and the Borel measurable optimal control $\varphi^*$ exists.

We shall prove the statement of all $t=T-2,\ldots,0$ by backward induction. Recall from Proposition~\ref{prop:borel}, the stochastic kernel $Q_{T-2}(\ \cdot\ |y,a,F)$ is Borel measurable. Also, the function $V_{T-1}$ is u.s.a.. Therefore,
$$
v_{T-2}(y,a,F)=\int_{E_Y}V_{T-1}(y_{T-1})Q_{T-2}(dy_{T-1}|y,a,F)
$$
is u.s.a.. By using a similar argument as above, the function
$$
v_{T-2}(y,a,F)=\int_{E_Y}V_{T-1}(\mathbf{G}(T-2,y,a,z))dF(z)
$$
is l.s.c.. The rest of the proof follows analogously.
\end{proof}

Finally, we show that the problem \eqref{eq:nar} will be solved by the adaptive robust Bellman equations \eqref{eq:bellman}.
To this end, we introduce the set $\cA_t=\{\varphi_{t:T-1}, t\in\cT'\}$, and provide the following technical results for preparation.

\begin{lemma}
For every $t\in\cT'$, and any $\varphi\in\cA_t$, the function
$$
\sup_{\bQ\in\cQ^{\varphi}_{y_t,t}}\bE_{\bQ}[\ell(X_T)]
$$
is upper semianalytic in $y_t$.
\end{lemma}

The proof for this lemma is a direct modification of Theorem~\ref{thm:selector} and hence we omit it here. Such result ensures that the mentioned function is measurable and can be integrated. Now we are ready to present the solution of the adaptive robust control problem.

\begin{theorem}\label{thm:solution}
For every $t\in\cT'$, and any $y_t\in E_Y$, we have
\begin{align*}
V_t(y_t)=\inf_{\varphi\in\cA_t}\sup_{\bQ\in\cQ^{\varphi}_{y_t,t}}\bE_{\bQ}[\ell(X_T)].
\end{align*}
Moreover, with $\varphi^*_t$ and $\psi^*_t$, $t\in\cT'$, in Theorem~\ref{thm:selector}, we get
\begin{align*}
\inf_{\varphi\in\cA_t}\sup_{\bQ\in\cA^\varphi_{y_t,t}}\bE_{\bQ}[\ell(X_T)]=\bE_{\bQ^{\varphi^*_{t:T-1},\psi^*_{t:T-1}}_{y_t,t}}[\ell(X_T)].
\end{align*}
\end{theorem}

\begin{proof}
We prove the result via backward induction in $t=T-1,\ldots,1,0$.

First, for $t=T-1$ and $y_{T-1}\in E_Y$, we have
\begin{align*}
\inf_{\varphi\in\cA_{T-1}}\sup_{\bQ\in\cQ^{\varphi}_{y_{T-1},T-1}}\bE_{\bQ}[\ell(X_T)]&=\inf_{a\in A}\sup_{F\in\cC^{\alpha}_{T-1}(y_{T-1})}\int_{E_Y}V_T(y_T)Q_{T-1}(y_T|y_{T-1},a,F)=V_{T-1}(y_{T-1}).
\end{align*}
Next, for $t=T-2,\ldots,0$ and $y_t\in E_Y$, by induction
\begin{align*}
\inf_{\varphi\in\cA_t}\sup_{\bQ\in\cQ^{\varphi}_{y_t,t}}\bE_{\bQ}[\ell(X_T)]&=\inf_{(\varphi_t,\varphi_{t+1:T-1})\in\cA_t}\sup_{F\in\cC^{\alpha}_t(y_t)}\int_{E_Y}\sup_{\bQ\in\cQ^{\varphi_{t+1:T-1}}_{y_{t+1},t+1}}\bE_{\bQ}[\ell(X_T)]Q_t(dy_{t+1}|y_t,\varphi_t(y_t),F)\\
&\geq\inf_{(\varphi_t,\varphi_{t+1:T-1})\in\cA_t}\sup_{F\in\cC^{\alpha}_t(y_t)}\int_{E_Y}V_{t+1}(y_{t+1})Q_t(dy_{t+1}|y_t,\varphi_t(y_t),F)\\
&=\inf_{a\in A}\sup_{F\in\cC^{\alpha}_t(y_t)}\int_{E_Y}V_{t+1}(y_{t+1})Q_t(dy_{t+1}|y_t,a,F)=V_t(y_t),
\end{align*}
where the inequality is due to that
$$
\sup_{\bQ\in\cQ^{\varphi_{t+1:T-1}}_{y_{t+1},t+1}}\bE_{\bQ}[\ell(X_T)]\geq\inf_{\varphi\in\cA_{t+1}}\sup_{\bQ\in\cQ^{\varphi}_{y_{t+1},t+1}}\bE_{\bQ}[\ell(X_T)]=V_{t+1}(y_{t+1}).
$$
On the other hand, for any $\varepsilon>0$, let $\varphi^{\varepsilon}_{t+1:T-1}\in\cA_{t+1}$ be an $\varepsilon$-optimal control starting at time $t+1$. We get
$$
\sup_{\bQ\in\cQ^{\varphi^{\varepsilon}_{t+1:T-1}}_{y_{t+1},t+1}}\bE_{\bQ}[\ell(X_T)]\leq\inf_{\varphi\in\cA_{t+1}}\sup_{\bQ\in\cQ^{\varphi}_{y_{t+1},t+1}}\bE_{\bQ}[\ell(X_T)]=V_{t+1}(y_{t+1})+\varepsilon.
$$
It is followed by
\begin{align*}
\inf_{\varphi\in\cA_t}\sup_{\bQ\in\cQ^{\varphi}_{y_t,t}}\bE_{\bQ}[\ell(X_T)]&=\inf_{(\varphi_t,\varphi_{t+1:T-1})\in\cA_t}\sup_{F\in\cC^{\alpha}_t(y_t)}\int_{E_Y}\sup_{\bQ\in\cQ^{\varphi_{t+1:T-1}}_{y_{t+1},t+1}}\bE_{\bQ}[\ell(X_T)]Q_t(dy_{t+1}|y_t,\varphi_t(y_t),F)\\
&\leq\inf_{(\varphi_t,\varphi_{t+1:T-1})\in\cA_t}\sup_{F\in\cC^{\alpha}_t(y_t)}\int_{E_Y}\sup_{\bQ\in\cQ^{\varphi^{\varepsilon}_{t+1:T-1}}_{y_{t+1},t+1}}\bE_{\bQ}[\ell(X_T)]Q_t(dy_{t+1}|y_t,\varphi_t(y_t),F)\\
&\leq\inf_{a\in A}\sup_{F\in\cC^{\alpha}_t(y_t)}\int_{E_Y}V_{t+1}(y_{t+1})Q_t(dy_{t+1}|y_t,a,F)+\varepsilon\\
&=V_t(y_t)+\varepsilon.
\end{align*}
Since $\varepsilon$ is arbitrary, we obtain
$$
\inf_{\varphi\in\cA_t}\sup_{\bQ\in\cQ^{\varphi}_{y_t,t}}\bE_{\bQ}[\ell(X_T)]\leq V_t(y_t).
$$
Hence, we have
$$
\inf_{\varphi\in\cA_t}\sup_{\bQ\in\cQ^{\varphi}_{y_t,t}}\bE_{\bQ}[\ell(X_T)]= V_t(y_t).
$$
To see that $\varphi^*$ and $\psi^*$ in Theorem~\ref{thm:selector} solve the adaptive robust control problem, we just need to note that for every $t\in\cT'$
\begin{align*}
V_t(y_t)&=\int_{E_Y}V_{t+1}(y_{t+1})Q_t(dy_{t+1}|y_t,\varphi^*_t(y_t),\psi^*_t(y_t))\\
&=\int_{E_Y}\int_{E_Y}V_{t+2}(y_{t+2})Q_{t+1}(dy_{t+2}|y_t,\varphi^*_t(y_{t+1}),\psi^*_{t+1}(y_{t+1}))Q_t(dy_{t+1}|y,\varphi^*_t(y_t),\psi^*_t(y_t))\\
&=\int_{E_Y}\cdots\int_{E_Y}V_T(y_T)\prod_{s=t}^{T-1}Q_{s+1}(dy_s|y_s,\varphi^*_s(y_s),\psi^*_s(y_s))\\
&=\bE_{\bQ^{\varphi^*_{t:T-1},\psi^*_{t:T-1}}_{y_t,t}}[\ell(X_T)],
\end{align*}
where the above $\psi^*_s(y_s)=\psi^*_s(y_s,\varphi^*_s(y_s))$, $t\in\cT'$, can be viewed as a composition of universally measurable functions and therefore universally measurable.
\end{proof}

\subsection{Convergence Analysis}\label{sec:convergence}

A nice property of the combination of Wasserstein metric and adaptive robust control is that convergence analysis can be done in such framework very easily. As shown in Theorem~\ref{thm:selector} and \ref{thm:solution}, to deal with \eqref{eq:nar} one employs the dynamic programming principle and solves the following Bellman equation
\begin{align*}
V_t(y)=\inf_{a\in A}\sup_{F\in\cC^\alpha_t(y)}\bE_F[V_{t+1}(\mathbf{G}(t,y,a,Z_{t+1}))], \quad t\in\cT'.
\end{align*}
According to \cite[Theorem 2]{BDOW2021}, by assuming $V$ and $S$ to be differentiable w.r.t. $x$, and denoting
$$
V^a_t(y)=\sup_{F\in\cC^\alpha_t(y)}\bE_F[V_{t+1}(\mathbf{G}(t,y,a,Z_{t+1}))],
$$
we get that
\begin{align}\label{eq:derivative}
V^a_t(y)=\bE_{\widehat{F}_t}[V_{t+1}(\mathbf{G}(t,y,a,Z_{t+1}))]+\frac{Q^H_t(1-\alpha)}{\sqrt{t_0+t}}\bE_{\widehat{F}_t}\left[\left|\frac{\partial}{\partial x}V_{t+1}(\mathbf{G}(t,y,a,Z_{t+1}))\right|\right]+o\left(\frac{1}{\sqrt{t_0+t}}\right).
\end{align}
For any given state $y=(x,\hat{f})\in E_Y$, denote by $z_{-t_0+1:0}$ the historical sample points that generate $\widehat{F}_0$, and let $z_{1:t}$ be the observations of $Z$ such that $\hat{f}(z)=\frac{\sum_{i=-t_0+1}^t\1_{\{z_i<z\}}}{t_0+t}$. The following expectation is computed as 
$$
\bE_{\widehat{F}_t}[V_{t+1}(\mathbf{G}(t,y,a,Z_{t+1}))]=\frac{1}{t_0+t}\sum_{i=-t_0+1}^tV_{t+1}(\mathbf{G}(t,y,a,z_i)),
$$
which is the sample mean of the random variable $V_{t+1}(\mathbf{G}(t,y,a,Z_{t+1}))$ given sample $z_{-t_0+1:t}$. By central limit theorem, we obtain that the convergence speed of $\bE_{\widehat{F}_t}[V_{t+1}(\mathbf{G}(t,y,a,Z_{t+1}))]$ to the expectation $\bE_{F^*}[V_{t+1}(\mathbf{G}(t,y,a,Z_{t+1}))]$ is asymptotically of order $\frac{1}{\sqrt{t_0+t}}$. Thus, as $t$ increases the adaptive robust control problem converges to the control problem without uncertainty and the convergence speed is of order $\frac{1}{\sqrt{t_0+t}}$. Moreover, we get by using the Chebyshev inequality that
\begin{align}\label{eq:chebyshev}
\bP\left(\left|\bE_{\widehat{F}_t}[V_{t+1}(\mathbf{G}(t,y,a,Z_{t+1}))]-\overline{\mu}^*_V\right|>\varepsilon\right)\leq\frac{\text{Var}\left(\bE_{\widehat{F}_t}[V_{t+1}(\mathbf{G}(t,y,a,Z_{t+1}))]\right)}{(t_0+t)\varepsilon^2},
\end{align}
where $\overline{\mu}^*_V=\bE_{F^*}[V_{t+1}(\mathbf{G}(t,y,a,Z_{t+1}))]$. Inequality \eqref{eq:chebyshev} implies that the first term on the right hand side in \eqref{eq:derivative} has a high probability of being close to $\bE_{F^*}[V_{t+1}(\mathbf{G}(t,y,a,Z_{t+1}))]$. For example, taking $\varepsilon=\frac{1}{\sqrt{t_0+t}}$, then \eqref{eq:chebyshev} implies
\begin{align*}
\bP\left(\left|\bE_{\widehat{F}_t}[V_{t+1}(\mathbf{G}(t,y,a,Z_{t+1}))]-\overline{\mu}^*_V\right|>\frac{1}{\sqrt{t_0+t}}\right)\leq\text{Var}\left(\bE_{\widehat{F}_t}[V_{t+1}(\mathbf{G}(t,y,a,Z_{t+1}))]\right).
\end{align*}
Clearly, the above probability will continue to decrease as $t$ increases.
Note that with a further assumption given in the Cramer's Theorem:
\begin{align}\label{eq:exponential}
\int_{\bR}e^{\theta z}F^*(dz)<\infty, \quad \forall \theta\in\bR,
\end{align}
we have
$$
\lim_{t\to\infty}\frac{1}{t+t_0}\log\bP\left(\left|\bE_{\widehat{F}_t}[V_{t+1}(\mathbf{G}(t,y,a,Z_{t+1}))]-\overline{\mu}^*_V\right|>\varepsilon\right)=-\sup_{\theta\in\bR}((\overline{\mu}^*_V+\varepsilon)\theta-\log\bE_{F^*}[e^{\theta Z_{t+1}}]).
$$
As a result, the probability that $\bE_{\widehat{F}_t}[V_{t+1}(\mathbf{G}(t,y,a,Z_{t+1}))]$ deviates from the true value function for more than $\varepsilon$ has an exponential decay in time with speed $-\sup_{\theta\in\bR}((\overline{\mu}^*_V+\varepsilon)\theta-\log\bE_{F^*}[e^{\theta Z_{t+1}}])$ which is an obvious improvement over \eqref{eq:chebyshev}.
In summary, if assuming \eqref{eq:exponential} and using $\cC^{\alpha}_t(y)$ as the uncertainty set, even though the overall convergence speed of $V_t$ is still of order $\frac{1}{\sqrt{t_0+t}}$, we obtain a more accurate value function compared to the true one.

Based on \eqref{eq:derivative}, we can compare the adaptive robust framework to the static robust setup in a qualitative manner.
For the latter, the uncertainty set is fixed for all $t\in\cT'$, and we denote it by $\cP_0$. To solve the static robust control problem, one also utilizes the dynamic programming principle and solves
\begin{align*}
\widetilde{V}_t(y)=\inf_{a\in A}\sup_{F\in\cP_0}\bE_F[\widetilde{V}_{t+1}(\mathbf{G}(t,y,a,Z_{t+1}))], \quad t\in\cT',
\end{align*}
where $\widetilde{V}$ is the corresponding value function.
We consider the set $\cP_0$ defined as $\cB_{\delta}(\widehat{F}_0)$ which is a Wasserstein ball arond $\widehat{F}_0$ with radius $\delta$.
We also define a preference relation between value functions via
\begin{align*}
V_t(y)\succeq\widetilde{V}_t(y) \iff \sup_{F\in\cC^\alpha_t(y)}\bE_F[V_{t+1}(\mathbf{G}(t,y,a,Z_{t+1}))] \leq \sup_{F\in\cB_\delta(\widehat{F}_0)}\bE_F[\widetilde{V}_{t+1}(\mathbf{G}(t,y,a,Z_{t+1}))],
\end{align*}
for any $a\in A$. Next, suppose that $F^*\in\cP^{\mathrm{o}}_0$ which is the interior of the set $\cP_0$. For large $T$ and $t$, we have $d_{W,1}(\widehat{F}_t,F^*)<d_{W,1}(F^*,\partial\cP_0)$ with high probability, where $d_{W,1}(F^*,\partial\cP_0)$ is the Wasserstein distance from between $F^*$ and the closest point on the boundary of $\cP_0$. Consequently, $\cC^\alpha_t(Y_t)\subset\cP_0$ with high probability, and loosely speaking we get
\begin{align}\label{eq:pref}
V_t(y)\succeq\widetilde{V}_t(y),
\end{align}
asymptotically. Note that such discussion is rather qualitative since it is not easy to compute $\bP(\cC^\alpha_t(Y_t)\subset\cP_0)$ and prove \eqref{eq:pref} rigorously. Nevertheless, we argue that adaptive robust framework is more preferrable than static robust.

For a more quantitative analysis, we assume that $\cP_0=\cC^{\alpha}_0(y_0)$. Similarly to \eqref{eq:derivative}, we have
\begin{align*}
\widetilde{V}^a_t(x)=\bE_{\widehat{F}_0}[\widetilde{V}_{t+1}(S(x,a,Z_{t+1}))]+\frac{Q^H_0(1-\alpha)}{\sqrt{t_0}}\bE_{\widehat{F}_0}\left[\left|\frac{\partial}{\partial x}\widetilde{V}_{t+1}(S(x,a,Z_{t+1}))\right|\right]+o(1).
\end{align*}
It is obvious that the right hand side of the above equality does not converge with respect to $t$. As a result, the static robust framework will produce strategies that in general distant from the optimal strategies without uncertainty. Such strategies behave very conservatively while adaptive robust has a better balance between being aggressive and conservative due to the embedded learning feature. In view of such, the adaptive robust methodology is more favorable compared to the static robust framework which offers no convergence to the true optimization problem.

Note that discussions in this section are possible since we are using the Wasserstein metric to define the uncertainty sets. Similar analysis could be done when utilizing the Kullback-Leibler divergence but stronger assumptions on the considered probability distributions are required.

\section{Nonparametric Adaptive Robust Utility Maximization}\label{sec:numerical}

In this section, we consider a utility maximization problem under model uncertainty and we will solve it under the nonparametric adaptive robust framework. To this end, we take $X$ to be the investor's wealth process. Any portfolio includes two assets: a banking account with 1-period return $1+r$, where $r$ is the interest rate and fixed throughout, and a stock with i.i.d. log-return $Z_t$, $t\in\cT''$, of which the distribution $F^*$ is unknown. For each $t\in\cT'$, denote by $\varphi_t$ the ratio of the wealth invested in the stock. We rule out leverage and short selling, so $\varphi_t$ takes values in $A=[0,1]$. Imposing the self-financing strategy, and given $X_0=x_0>0$, the dynamics of $X$ is given by
$$
X_{t+1}=X_t((1-\varphi_t)(1+r)+\varphi_te^{Z_{t+1}}), \quad t\in\cT'.
$$
Take $n=1$, and the function $S$ is defined on $\bR\times A\times \bR$. The prices of the risky asset are observable and thus the return process $Z$ of the risky asset is also observable. We will use the observations of $Z$ to construct the empirical distribution iteratively as in \eqref{eq:f_map}. Then, we build the $\alpha$-uncertainty sets for the distribution $F$ of $Z$ according to \eqref{eq:uncertainty_set}.
Next, by taking $\ell(x)=\frac{e^{-\eta x}-1}{\eta}$ for some $\eta>0$,
we formulate the nonparametric adaptive robust utility maximization problem as
$$
\inf_{\varphi\in\cA}\sup_{\bQ\in\cQ^{\alpha}_{y_0}}\bE_{\bQ}[\ell(X_T)],
$$
where $y_0=(x_0,\widehat{F}_0)$ such that $\widehat{F}_0$ is the initial guess of $F^*$. Note that the funtion $\ell$ is bounded and we are equivalently dealing with
\begin{align}\label{eq:utility}
\sup_{\varphi\in\cA}\inf_{\bQ\in\cQ^{\alpha}_{y_0}}\bE_{\bQ}\left[\frac{1-e^{-\eta X_T}}{\eta}\right]
\end{align}
which is a maximization problem of the exponential utility function. Due to Theorem~\ref{thm:solution}, we will solve the following Bellman equations to get the solution of \eqref{eq:utility}.
\begin{align}\label{eq:max}
V_T(y)&=\frac{1-e^{-\eta x}}{\eta},\nonumber\\
V_t(y)&=\sup_{a\in A}\inf_{F\in\cC^{\alpha}_t(y)}\int_{E_Y}V_{t+1}(y_{t+1})Q_t(dy_{t+1}|y,a,F), \quad t\in\cT'.
\end{align}
Moreover, by applying Theorem~\ref{thm:selector}, we get that the optimal trading strategies and worst case models exist which are optimizers of \eqref{eq:max}.

\begin{remark}
Several types of utility functions satisfy the assumptions in Theorem~\ref{thm:selector} so that the corresponding optimal trading strategies and worst case models exist, and the adaptive robust control problem can be solved by utilizing the dynamic programming principle. Another example of such utility functions is the power utility $\frac{x^{1-\eta}-1}{1-\eta}$ where $\eta>1$.
\end{remark}

Note that the loss function $\ell(x)=\frac{e^{-\eta x}-1}{\eta}$ is not only bounded from below but actually bounded. Here we provide the following technical result regarding the corresponding value functions.

\begin{proposition}\label{prop:lsc}
The value function $V_t(y)$ as in \eqref{eq:bellman} is lower semicontinuous for every $t\in\cT'$.
\end{proposition}
\begin{proof}
The function $V_T(y)=\frac{1-e^{-\eta x}}{\eta}$ is clearly continuous and hence u.s.c.. Because $\mathbf{G}(T-1,y,a,z)$ is continuous in $(y,a,z)$, $V_T(\mathbf{G}(T-1,y,a,z))$ is l.s.c. in $(y,a,z)$. Moreover,
$$
v_{T-1}(y,a,F)=\int_{\bR}V_T(\mathbf{G}(T-1,y,a,z))dF(z)
$$
is l.s.c. due to that $W_T$ is bounded.

Consider the set $D=\bigcup_{(y,a)\in E_Y\times A}\cC^{\alpha}_{T-1}(y)$, and define the function
$$
\check{v}_{T-1}(y,a,F)=
\begin{cases}
v_{T-1}(y,a,F) \quad &\text{if } (y,a,F)\in D,\\
\infty \quad &\text{otherwise.}
\end{cases}
$$
For any $c\in\bR$, we have
\begin{align*}
&\left\{(y,a,F)\in E_Y\times A\times\cP_1(\bR)\mid \check{v}_{T-1}(y,a,F)\leq c\right\}\\
=&\left\{(y,a,F)\in E_Y\times A\times\cP_1(\bR)\mid v_{T-1}(y,a,F)\leq c\right\}\cap D.
\end{align*}
Since $v_{T-1}$ is l.s.c., and $D$ is closed, then the set $\left\{(y,a,F)\in E_Y\times A\times\cP_1(\bR)\mid \check{v}_{T-1}(y,a,F)\leq c\right\}$ is closed and $\check{v}_{T-1}(y,a,F)$ is l.s.c.. Next, for any $c\in\bR$,
\begin{align*}
&\left\{(y,a)\in E_Y\times A\mid \inf_{F\in\cC^{\alpha}_{T-1}(y)}v_{T-1}(y,a,F)\leq c\right\}\\
=&\left\{(y,a)\in E_Y\times A\mid \inf_{F\in\cP_1(\bR)}\check{v}_{T-1}(y,a,F)\leq c\right\}.
\end{align*}
Fix $(y,a)$ and let $\{F_n,n>0\}\subset\cP_1(\bR)$ be such that
$$
\check{v}(y,a,F_n)\downarrow\widetilde{v}_{T-1}(y,a):=\inf_{F\in\cP_1(\bR)}\check{v}_{T-1}(y,a,F).
$$
By definition of $\check{v}_{T-1}$, we know for large enough $n$, $F_n\in\cC^{\alpha}_{T-1}(y)$ which is a weakly compact set. Then, there exists $F^*$ such that $\check{v}_{T-1}(y,a,F^*)=\widetilde{v}_{T-1}(y,a)$. Let $\{(y_n,a_n),n>0\}$ be a sequence that converges to some $(y_0,a_0)$. We choose a sequence $\{F_n,n>0\}\subset\cP(\bR)$ such that $\check{v}_{T-1}(y_n,a_n,F_n)=\widetilde{v}_{T-1}(y_n,a_n)$. Obviously, for each $n>0$, $F_n\in\cC^{\alpha}_{T-1}(y_n)$. Due to the fact that $\{y_n,n>0\}$ converges to $y_0$, the set $\widetilde{D}=\bigcup_n\cC^{\alpha}_{T-1}(y_n)$ is bounded. Hence, there exists $F'\in\widetilde{D}$ and $\delta>0$ such that $\widetilde{D}\subset\cB_\delta(F')$ where the latter is a Wasserstein ball around $F'$ with radius $\delta$.

Now we consider the topology consistent with the weak convergence for the argument $F$ in the function $\check{v}_{T-1}(y,a,F)$. In such case, $\check{v}_{T-1}$ is still l.s.c.. There exists a subsequence $(y_{n_k},a_{n_k},F_{n_k})$, $k>0$, such that
$$
\liminf_{n\to\infty}\check{v}_{T-1}(y_n,a_n,F_n)=\lim_{k\to\infty}\check{v}_{T-1}(y_{n_k},a_{n_k},F_{n_k}).
$$
As $\cB_{\delta}(F')$ is compact under the Prokhorov metric, there exists $F_0$ that is a limit point of $\{F_{n_k},n>0\}$. We obtain
\begin{align*}
\liminf_{n\to\infty}\widetilde{v}_{T-1}(y_n,a_n)&=\liminf_{n\to\infty}\check{v}_{T-1}(y_n,a_n,F_n)=\lim_{k\to\infty}\check{v}_{T-1}(y_{n_k},a_{n_k},F_{n_k})\\
&\geq\check{v}_{T-1}(y_0,a_0,F_0)\geq\widetilde{v}_{T-1}(y_0,a_0).
\end{align*}
This shows that $\widetilde{v}_{T-1}(y,a)$ is l.s.c.. Next, take set $\cO=E_Y\times (0,1)$ and such set is open. The $y$ section of $\cO$ is the interval $(0,1)$. By \cite{BS1978}
$$
\widehat{V}_{T-1}(y)=\sup_{a\in(0,1)}\widetilde{v}_{T-1}(y,a)
$$
is l.s.c.. Note that $A=[0,1]$ is the closure of $(0,1)$, thus
$$
V_{T-1}(y)=\sup_{a\in A}\widetilde{v}_{T-1}(y,a)=\sup_{a\in(0,1)}\widetilde{v}_{T-1}(y,a)=\widehat{V}_{T-1}(y),
$$
and $V_{T-1}(y)$ is l.s.c..
Following the backward induction for $t=T-2,\ldots,0$, the proof is complete.
\end{proof}

Proposition~\ref{prop:lsc} is of great importance for numerical computation of Bellman equations~\eqref{eq:max}. As in \cite{KENA2019}, when $V_{t+1}$ is l.s.c., for any fixed $(y,a)\in E_Y\times A$, the inner optimization problem can be solved as follows
\begin{align*}
\inf_{F\in\cC^{\alpha}_t(y)}\int_{E_Y}V_{t+1}(y_{t+1})Q_t(dy_{t+1}|y,a,F)&=\inf_{F\in\cC^{\alpha}_t(y)}\int_{\bR}V_{t+1}(\mathbf{G}(t,y,a,z))dF(z)\\
&=\sup_{\gamma\geq0}\left\{\bE_{\widehat{F}}\left[V^{\gamma}_{t+1}(\mathbf{G}(t,y,a,Z_{t+1}))\right]-\frac{\gamma Q^H_t(1-\alpha)}{\sqrt{t_0+t}}\right\},
\end{align*}
where $V^{\gamma}_{t+1}(\mathbf{G}(t,y,a,Z_{t+1}))=\inf_{z\in\bR}\left\{V_{t+1}(\mathbf{G}(t,y,a,z))+\gamma|z-Z_{t+1}|\right\}$, and $y=(x,\widehat{F})$. With such results in hand, the Bellman equation \eqref{eq:max} becomes
\begin{align}
V_T(y)&=\frac{1-e^{-\eta x}}{\eta},\nonumber\\
V_t(y)&=\sup_{a\in A,\gamma\geq0}\left\{\bE_{\widehat{F}}[V^{\gamma}_{t+1}(\mathbf{G}(t,y,a,Z_{t+1}))]-\frac{\gamma Q^H_t(1-\alpha)}{\sqrt{t_0+t}}\right\},\label{eq:max2}\\
V^{\gamma}_{t+1}(\mathbf{G}(t,y,a,Z_{t+1}))&=\inf_{z\in\bR}\{V_{t+1}(\mathbf{G}(t,y,a,z))+\gamma|z-Z_{t+1}|\}.\label{eq:max3}
\end{align}
In the sequel, we will dicuss the challenges in the numerical computation of \eqref{eq:max2} and \eqref{eq:max3} and explain our algorithm for dealing with such problem.

\subsection{Algorithm}
In this practice, we will mainly follow the idea represented in \cite{CL2019} and propose a similar numerical scheme that uses regression Monte Carlo and GP surrogates to solve the Bellman equations \eqref{eq:max2} and \eqref{eq:max3}. Then, we analyze the performance of the obtained optimal control on out-of-sample paths by simulating the realized terminal utility and estimating the expected utility.

Towards this end, we begin with discretizing the state space by choosing $y^i_t=(x^i_t,\widehat{F}^i_t)\in E_Y$, $i=1,\ldots,N$, $t\in\cT$.
These $y^i_t$'s are called design points.
Then, we solve the equation \eqref{eq:max2} for the design points $y=y^i_t$, $i=1,\ldots,N$, $t=T,T-1\ldots,0$.
One of the main tasks in the numerical algorithm is computing $\bE_{\widehat{F}}[V^{\gamma}(\mathbf{G}(t,y^i_t,a,Z_{t+1})]$ for $i=1,\ldots,N$, $t\in\cT$. 
In view of $\widehat{F}^i_t$ being an empirical distribution and assuming that
$$
\widehat{F}^i_t(z)=\frac{1}{t_0+t}\sum_{j=-t_0+1}^t\1_{z^i_j\leq z},
$$
we have
\begin{align}\label{eq:expectation}
\bE_{\widehat{F}^i_t}[V^{\gamma}_{t+1}(\mathbf{G}(t,y^i_t,a,Z_{t+1}))]=\frac{1}{t+t_0}\sum_{j=-t_0+1}^tV^{\gamma}_{t+1}(\mathbf{G}(t,y^i_t,a,z^i_j)).
\end{align}
\begin{remark}
In our current setup, $\widehat{F}_0$ is defined to be an empirical distribution constructed from historical data prior to the beginning of the investment, but it does not have to be. For example, there are estimation techniques that produce continuous prior distribution $\widehat{F}_0$ (cf. perturbed empirical distribution), and in such case Monte Carlo method will be needed to compute $\bE_{\widehat{F}}[V^{\gamma}_{t+1}(\mathbf{G}(t,y,a,Z_{t+1}))]$ due to the fact that $\widehat{F}$ is no longer a discrete distrubtion anymore.
\end{remark}
Since the value function $V_{t+1}$, and in turn $V^{\gamma}_{t+1}$, cannot be computed analytically, we will need a regression model for $V_{t+1}$ so that we can estimate the right hand side of \eqref{eq:expectation}.
The general strategy is then, for every $t\in\cT'$, we use $(y^i_{t+1},V_{t+1}(y^i_{t+1}))$, $i=1,\ldots,N$, called training points to build a regression model for $V_{t+1}$, and use it to evaluate $V^{\gamma}_{t+1}$.
Thus, we have an optimize--train--optimize loop in our algorithm.
The state component $\widehat{F}^i_t$ is a probability distribution which is infinitely dimensional, or can be equivalently replaced by the vector $z^i_{-t_0+1:t}$.
In both cases, we are dealing with a high dimensional problem and facing the challenge of ``curse of dimensionality''.
Due to such reason, the traditional grid-based method for choosing the design points $y^i_t$, $i=1,\ldots,N$, $t\in\cT$, will be inefficient.
To overcome this difficulty, we use the idea of randomized control so that we can focus on the points in the state space that are likely to be visited by the state process $Y$.
In particular, for $t\in\cT'$, given the design points $y^i_t,\ldots,y^N_t$, we will uniformly generate $a^1,\ldots,a^N$ from $A$ and use them to update $y^1_t,\ldots,y^N_t$ to $y^1_{t+1},\ldots,y^N_{t+1}$, respectively, according to
$$
y^i_{t+1}=\mathbf{G}(t,y^i_t,a^i,Z^i_{t+1}), \quad i=1,\ldots,N,
$$
where $Z^i_{t+1}$ is the simulated random noise.

Next, we discuss the choise of regression model for the value function $V_{t+1}$ in detail.
From above we see that for each $t\in\cT'$, $V_{t+1}$ can be viewed as a function of $(x_{t+1},z_{-t_0+1:t+1})$ where $z_{-t_0+1:t+1}$ yields the empirical distribution $\widehat{F}_{t+1}$.
Therefore, it is natural to regress $V_{t+1}$ against $(x_{t+1},z_{-t_0+1:t+1})$ instead of $(x_{t+1},\widehat{F}_{t+1})$.
Such treatment will reduce an infinite dimensional problem to a finite one.
However, note that $(x_{t+1},z_{-t_0+1:t+1})$ has a dimension of $t_0+t+2$ and to regress $V_{t+1}$ against such high dimensional input requires an enormous amount of training points $(x^i_{t+1},z^i_{-t_0+1:t+1})$, $i=1,\ldots,N$, so that we can obtain an accurate regression model for $V_{t+1}$.
Hence, solely for the regression purpose, we will approximate $\widehat{F}_{t+1}$ with its first $d$ moments denote by $m^1_{t+1},\ldots,m^d_{t+1}$, and regress $V_{t+1}$ against $(x_{t+1},m^1_{t+1},\ldots,m^d_{t+1})$.
By doing so, we effectively approximate a $t_0+t+2$-dimensional function with a $d+1$-dimensional regression model.
Since the moments of a distrubtion capture the features of the distribution quite well, our strategy is a sound way to reduce the dimension of the problem that we are facing.
To this end, we propose to use the GP surrogate to build regression models for $V_{t+1}$, $t\in\cT'$.
Gaussian process is a popular tool in machine learning that is suitable for dealing with regression problem with mid-range dimensions.
It produces nonparametric functional approximations of functions by utilizing the location information of the function input. Namely, for some ``usual'' function $g$, if $\|u_1-u_2\|$ is small, then a GP user assumes that $\|g(u_1)-g(u_2)\|$ should be relatively small as well. Recall that from Theorem~\ref{thm:selector} and Proposition~\ref{prop:lsc}, we immediately get the following result.
\begin{corollary}
For every $t\in\cT$, and $V_t$ defined in \eqref{eq:max}, $V_t$ is a continuous function on $E_Y$.
\end{corollary}
Hence, GP is the ideal tool for us to build the statistical surrogates for each $V_t$, $t\in\cT''$, so that we can proceed with the backward iteration and solve the Bellman equations.
To be more specific, we approximate each of the design points $y^i_{t+1}$, $i=1,\ldots,N$, by $\check{y}^i_{t+1}:=(x^i_{t+1},m^{i,1}_{t+1},\ldots,m^{i,d}_{t+1})$, and denote by $\check{V}_{t+1}$ the GP surrogate of $V_{t+1}$.
Then, in the context of GP regression, the values $\check{V}_{t+1}(\check{y}^i_{t+1})$, $i=1,\ldots,N$, are jointly normal distributed.
For any $y\in E_Y$, the predicted value $\check{V}_{t+1}(y)$ that approximates $V_{t+1}(y)$ is then computed as
$$
\check{V}_{t+1}(y)=(k(y,\check{y}^1_{t+1}),\ldots,k(y,\check{y}^N_{t+1}))[\mathbf{K}+\epsilon^2\mathbf{I}]^{-1}(V_{t+1}(\check{y}^1_{t+1}),\ldots,V_{t+1}(\check{y}^N_{t+1}))^\top,
$$
where $\mathbf{I}$ is the $N\times N$ identity matrix and entries of $\mathbf{K}$ has the form $\mathbf{K}_{ij}=k(\check{y}^i_{t+1},\check{y}^j_{t+1})$, $i,j=1,\ldots,N$. The function $k(\cdot,\cdot)$ is called the kernel function of the GP surrogate and in this project, we choose it from the Matern-5/2 family (cf. \cite{Genton2002}).
We fit $\check{V}_{t+1}$ to the training points $\{(\check{y}^i_{t+1},V_{t+1}(\check{y}^i_{t+1})),i=1,\ldots,N\}$ and during this process the hyperparameters inside of $k(\cdot,\cdot)$ will be estimated.
For a comprehensive discussion of the Gaussian process surrogates, we refer to the book \cite{RW2006}.

We summarize our algorithm for solving \eqref{eq:max2} and \eqref{eq:max3} as follows:

\begin{enumerate}
\item
(Assume that $V_{t+1}(\cdot)$ and $\varphi^*_{t+1}(\cdot)$ are computed (estimated) at design points $y^1_{t+1},\ldots,y^N_{t+1}$, $t\in\cT''$, and the GP surrogates $\check{V}_{t+1}$ and $\check{\varphi}^*_{t+1}$ \footnote{The GP surrogate $\check{\varphi}^*_{t+1}$ is the Gaussian process regression model constructed by using the training data $\{(\check{y}^i,\varphi^*_{t+1}(y^i)),i=1,\ldots,N\}$.} are fitted.)
\item
For time $t$, any $a\in A$, $\gamma>0$, $z\in\bR$, and each of the design points $\{y^i_t,i=1,\ldots,N\}\subset E_Y$, use the GP surrogate $\check{V}_{t+1}$ and command \verb|scipy.optimize.minimize_scalar| in the scipy package for \verb|Python| to compute
\begin{align*}
\check{V}^{\gamma}_{t+1}(\mathbf{G}(t,y^i_t,a,z)):=\inf_{z'\in\bR}\{\check{V}_{t+1}(\mathbf{G}(t,y^i_t,a,z))+\gamma|z'-z|)\}, \quad i=1,\ldots,N,
\end{align*}
and $\check{V}^{\gamma}_{t+1}$ is an approximation of $V^{\gamma}_{t+1}$.
\item
For time $t$, any $a\in A$, and each of the design points $\{y^i_t,i=1,\ldots,N\}\subset E_Y$, approximate $\bE_{\widehat{F}}[V^{\gamma}_{t+1}(\mathbf{G}(t,y^i_t,a,Z_{t+1}))]$ as
$$
\bE_{\widehat{F}}[V^{\gamma}_{t+1}(\mathbf{G}(t,y^i_t,a,Z_{t+1}))]\approx\frac{1}{t+t_0}\sum_{j=-t_0+1}^t\check{V}^{\gamma}_{t+1}(\mathbf{G}(t,y^i_t,a,z^i_j)).
$$
\item
Use the command \verb|scipy.optimize.minimize_scalar| to compute
$$
V^{(1)}(y^i_t,a)=-\inf_{\gamma\geq0}\left\{-\frac{1}{t+t_0}\sum_{j=-t_0+1}^t\check{V}^{\gamma}_{t+1}(\mathbf{G}(t,y^i_t,a,z^i_j))+\frac{\gamma Q^H_t(1-\alpha)}{\sqrt{t_0+t}}\right\}, \quad i=1,\ldots,N,
$$
and
$$
V_t(y^i_t)=-\inf_{a\in A}(-V^{(1)}(y^i_t,a)),
$$
where we also obtain the optimizer $\varphi^*_t(y^i_t)$, $i=1,\ldots,N$.
\item
Fit the GP surrogate $\check{V}_t$ by using $(\check{y}^i_t,V_t(y^i_t))$, $i=1,\ldots,N$, as the training points. Similarly, fit $\check{\varphi}^*_t$ by using $(\check{y}^i_t,\varphi_t(y^i_t))$, $i=1,\ldots,N$.
\item
Goto 1.: start the next recursion for $t-1$.
\end{enumerate}

To analyze the performance of the optimal control we obtain from solving the Bellman equations, we generate $N'$ forward simulated paths by starting with the initial state $y_0=(x_0,\widehat{F}_0)$ and applying the control $\check{\varphi}^*_t(\check{y}^i_t)$, $i=1,\ldots,N'$, to obtain the next-step state $y^i_{t+1}$ for $t\in\cT'$ according to
$$
y^i_{t+1}=\mathbf{G}(t,y^i_t,\check{\varphi}^*_t(\check{y}^i_t),Z^i_{t+1}).
$$
The corresponding forward Monte Carlo algorithm is summarized as
\begin{enumerate}
\item
Take $y^i_0\equiv(x_0,\widehat{F}_0)$, $i=1,\ldots,N'$.
\item
For $t=1,\ldots,T$, generate $Z^i_t$, $i=1,\ldots,N'$.
\item
Approximate $y^i_t$ as $\check{y}^i_t$ and use the GP surrogates to compute the control $a^i_t=\check{\varphi}_t(\check{y}^i_t)$, $i=1,\ldots,N'$, $t\in\cT'$.
\item
Update the states $y^i_{t+1}=\mathbf{G}(t,y^i_t,a^i_t,Z^i_{t+1})$, $i=1,\ldots,N'$, $t=0,\ldots,T-1$.
\item
Compute $\widehat{V}_0(y_0)=\frac{1}{N'}\sum_{i=1}^{N'}\frac{1-e^{-\eta x^i_T}}{\eta}$.
\end{enumerate}

The average $\widehat{V}_0(y_0)$ is then the Monte Carlo estimator of the expected utility. In addition, we are interested in the distribution of the utility
$$
\widehat{U}(y_0)=\left(\frac{1-e^{-\eta x^1_T}}{\eta},\ldots,\frac{1-e^{-\eta x^{N'}_T}}{\eta}\right)
$$
and the numerical results will be reported in the sequel.

\begin{figure}[h!]
   \centering

     \includegraphics[width=0.49\textwidth]{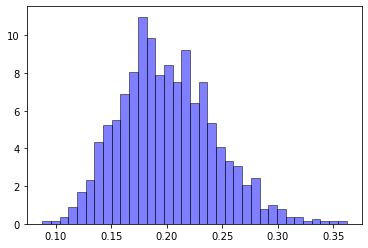}
     \includegraphics[width=0.49\textwidth]{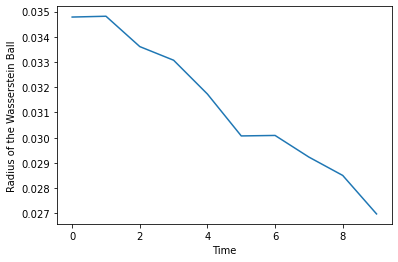}

     \caption{Left panel: Histogram of simulated $Q^H_0(1-\alpha)$ for $t_0=20$ and $\alpha=0.1$. Right panel: simulated path of the radius of $\cC^{\alpha}_t$.}

     \label{fig:plot1}
\end{figure}

\subsection{Numerical Results}

In this section, we apply the machine learning algorithm described above to some specific sets of parameters. We will compare the performance of nonparametric adaptive robust method to that of some other frameworks used to deal with model uncertainty. Theoretically, the optimal control is attained when there is no model uncertainty. We will also analyze the difference in performance between the cases of knowing the true model and having to estimate and learning the dynamics of the underlying stochastic process. To this end, we consider three types of investors: the one who knows the true model with terminal utility $\widehat{U}^{\text{TR}}(y_0)$ and expected utility $\widehat{V}_0^{\text{TR}}(y_0)$; the one that applies the nonparametric adaptive robust with terminal utility $\widehat{U}^{\text{AR}}(y_0)$ and expected utility $\widehat{V}_0^{\text{AR}}(y_0)$; finally, the one uses the static robust methods, meaning the corresponding uncertainty sets do not change with respect to the state and time. In particular, the static robust investor utilizes the nonparametric setup and builds the uncertainty set as a Wasserstein ball around the empirical distribution generated by historical data with sample size $t_0$. The terminal utility and expected terminal utility of the nonparametic static robust investor are $\widehat{U}^{\text{SR}}(y_0)$ and $\widehat{V}_0^{\text{SR}}(y_0)$, respectively.

\begin{table}[h!]
\centering
\renewcommand{\arraystretch}{1.3}
\begin{tabular}{c|ccc}
%\cline{2-10} \cline{2-10}
\hline
 \multicolumn{1}{c}{ }& \multicolumn{1}{||c}{AR} & \multicolumn{1}{|c}{TR} & \multicolumn{1}{|c}{SR} \\
  \hline
   $\widehat{V}_0$  & \multicolumn{1}{||c}{65.425570} & \multicolumn{1}{|c}{66.805075} & \multicolumn{1}{|c}{63.947066}\\
 var($\widehat{U}$)   & \multicolumn{1}{||c}{36.679199} & \multicolumn{1}{|c}{108.601415} & \multicolumn{1}{|c}{$6.451175\cdot 10^{-9}$}\\
 $q_{0.20}(\widehat{U})$  & \multicolumn{1}{||c}{59.682528} & \multicolumn{1}{|c}{58.740896} & \multicolumn{1}{|c}{63.947003}\\
 $q_{0.90}(\widehat{U})$  & \multicolumn{1}{||c}{72.937869} & \multicolumn{1}{|c}{78.899913} & \multicolumn{1}{|c}{63.947173}\\
 $\text{max}(\widehat{U})$  & \multicolumn{1}{||c}{82.953448} & \multicolumn{1}{|c}{90.773115} & \multicolumn{1}{|c}{63.947302}\\
 $\text{min}(\widehat{U})$  & \multicolumn{1}{||c}{46.192910} & \multicolumn{1}{|c}{26.307049} & \multicolumn{1}{|c}{63.946811}\\
 \hline
\end{tabular}
\bigskip
\caption{Mean, variance, 20\%-quantile, 90\%-quantile, maximum, and minimum of the out-of-sample terminal utility for the AR, TR and SR methods; $Q^H_0(1-\alpha)=0.199165$.}
\label{table:table1}
\end{table}

Note that we can easily modify the above algorithm to compute $\widehat{U}^{\text{TR}}$, $\widehat{V}^{\text{TR}}_0$, $\widehat{U}^{\text{SR}}$, and $\widehat{V}^{\text{SR}}_0$. In fact, by taking $\cC^{\alpha}_t(y)\equiv\cB_0(F^*)$ which is the Wasserstein ball around $F^*$ with 0 radius, we are able to compute $\widehat{U}^{\text{TR}}$ and $\widehat{V}^{\text{TR}}_0$. For $\widehat{U}^{\text{SR}}$ and $\widehat{V}^{\text{SR}}_0$, we take $\cC^{\alpha}_t(y)\equiv\cC^{\alpha}_0(y_0)$.

We choose the terminal time to be 1 year with $T=10$ time steps which means one unit of time is 0.1 year. The annual insterest rate is 0.02 so that $r=0.02/10=0.002$. Initial endowment is $x_0=100$. Some other parameters are $\alpha=0.1$, $\eta=0.01$, and $m=4$. The number of paths is $N=1000$ for nonparametric adaptive robust and 200 for other methods. The reason for such choice is that the state space of adaptive robust has dimension $m+1$ while the others have dimension 1.
For the sampling measure and test measure, we consider a Gaussian mixture model: with 40\% probability, $Z_t\sim N(0.06/10,0.4^2/10)$, and with 60\% probability, $Z_t\sim N(0.16/10,0.25^2/10)$. Recall that the parametric static robust investor assumes that $Z_t\sim N(\mu,\sigma^2)$ and constructs the confidence region for $\mu$ and $\sigma^2$. We will compute and compare the distributions of utlities among the mentioned four frameworks with the above choice of parameters for $t_0=20$.
We also want to point out that the behavior of the optimal strategies would depend on the simulated $Q^H_0(1-\alpha)$.
In this exercise, we present two cases with $Q^H_0(1-\alpha)\approx0.199165$ and $Q^H_0(1-\alpha)\approx0.092942$.
Note that among 1000 simulated paths, 0.199165 sits very closely to the average value of $Q^H_0(1-\alpha)$ which is 0.200395, and $0.092942$ is below the 1\% quantile which is 0.115721.
We refer to the left panel of Figure~\ref{fig:plot1} for the histogram of simulated $Q^H_0(1-\alpha)$.

\begin{table}[h!]
\centering
\renewcommand{\arraystretch}{1.3}
\begin{tabular}{c|ccc}
%\cline{2-10} \cline{2-10}
\hline
 \multicolumn{1}{c}{ }& \multicolumn{1}{||c}{AR} & \multicolumn{1}{|c}{TR} & \multicolumn{1}{|c}{SR} \\
  \hline
   $\widehat{V}_0$  & \multicolumn{1}{||c}{65.440839} & \multicolumn{1}{|c}{66.805075} & \multicolumn{1}{|c}{63.947067}\\
 var($\widehat{U}$)   & \multicolumn{1}{||c}{41.907675} & \multicolumn{1}{|c}{108.601415} & \multicolumn{1}{|c}{$6.776359\cdot 10^{-9}$}\\
 $q_{0.20}(\widehat{U})$  & \multicolumn{1}{||c}{59.575356} & \multicolumn{1}{|c}{58.740896} & \multicolumn{1}{|c}{63.946997}\\
 $q_{0.90}(\widehat{U})$  & \multicolumn{1}{||c}{73.363772} & \multicolumn{1}{|c}{78.899913} & \multicolumn{1}{|c}{63.947175}\\
 $\text{max}(\widehat{U})$  & \multicolumn{1}{||c}{85.322523} & \multicolumn{1}{|c}{90.773115} & \multicolumn{1}{|c}{63.947367}\\
 $\text{min}(\widehat{U})$  & \multicolumn{1}{||c}{45.40310} & \multicolumn{1}{|c}{26.307049} & \multicolumn{1}{|c}{63.946763}\\
 \hline
\end{tabular}
\bigskip
\caption{Mean, variance, 20\%-quantile, 90\%-quantile, maximum, and minimum of the out-of-sample terminal utility for the AR, TR and SR methods; $Q^H_0(1-\alpha)=0.092942$.}
\label{table:table2}
\end{table}

For $Q^H_0(1-\alpha)\approx0.199165$, comparison among AR, TR, and SR are reported in Table~\ref{table:table1}. Since TR knows the true model of the risky asset return, the corresponding strategy will be optimal and $\widehat{V}^{TR}_0$ will outperform any other optimal control provided by investors who do not know the true model. Nevertheless, AR does better in three indices of risky management: AR has lower variance, higher 20\% quantile, and minimum value of the simulated terminal utilities than TR. AR also beats SR quite significantly in regard to the mean, 90\% quantile and maximum value of the simulated terminal utility. In addition, by viewing the Figure~\ref{fig:histogram1}, we argue that AR produces wealth paths with more favorable distribution than TR. On the other hand, SR generates trivial optimal strategies similarly to the observations made in some earlier work (cf. \cite{BCCCJ2019}, \cite{CM2020}). By ignoring the numerical instability, the terminal wealth produced by SR is a constant 102.018 which means all the money is invested in the banking account.
With no surprises, as such a conservative control method, SR performs well in the department of risk management: it has apparent minimal variance, higher 20\% quantile and minimum value of the terminal utility compared to AR and TR.

\begin{figure}[h!]
   \centering

     \includegraphics[width=0.49\textwidth]{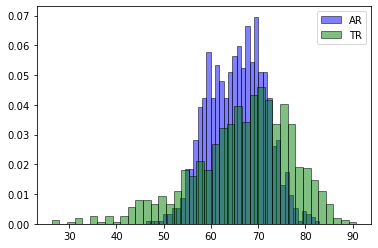}
     \includegraphics[width=0.49\textwidth]{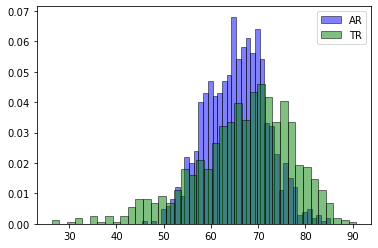}

     \caption{Histogram of the out-of-sample terminal utility $U$: AR vs TR. Left panel: $Q^H_0(1-\alpha)=0.199165$; right panel: $Q^H_0(1-\alpha)=0.092942$.}

     \label{fig:histogram1}
\end{figure}

For $Q^H_0(1-\alpha)\approx0.092942$, comparison of the performance of AR, TR, and SR on the same out-of-sample paths as in the previous case are reported in Table~\ref{table:table2}.
Since that $Q^H_0(1-\alpha)$ is smaller, the size of $\cC^{\alpha}_t$ along the simulated paths is in general smaller as a consequence.
Hence, we expect more aggressive strategies given by the robust approaches.
One needs to be aware that $Q^H_0(1-\alpha)\approx0.092942$ has an extremely low probability.
Thus, we expect the value of $Q^H_0(1-\alpha)$, and in turn the radius of $\cC^{\alpha}_t$ to be oscillating after $t=0$.
Nevertheless, we see from Table~\ref{table:table2} that there is an improvement of AR in this case.
Estimated expected utility $\widehat{V}^{\text{AR}}_0$ and the 90\% quantile of $\widehat{U}^{\text{AR}}$ are marginally larger than in the case of $Q^H_0(1-\alpha)=0.199165$.
Increase in the maximum value of $\widehat{U}^{\text{AR}}$ on the other hand is somewhat significant.
An unavoidable trade-off is that, even though only slightly, the strategy becomes more risky as the variance increases and 20\% quantile, as well as the minimum value, of $\widehat{U}^{\text{AR}}$ both decrease.
In line with our discussion, we also observe in Figure~\ref{fig:histogram1} that the distribution of $\widehat{U}^{\text{AR}}$ in the right panel has moderately larger tails on both left and right sides compared to that in the left panel.
Such change is expected to be more significant if the computation is done for larger $t_0$ and $T$.
To conclude, AR is more aggressive when the size of $\cC^{\alpha}_t$ is smaller but it is in general stable for our choice of parameters in the computation.
Regarding SR, we observe changes following a similar pattern as for AR.
However, such changes are so tiny and almost negligible.
Consequently, the computed SR strategies are considered as trivial and one needs to further reduce $Q^H_0(1-\alpha)$ in order to obtain a non-trivial SR optimal control.

The main argument for why SR being so conservative is that for relatively small historical data size $t_0$, the corresponding confidence region is usually too large.
On top of that, there is no shrinkage of the confidence region in static robust.
Hence, no matter at which time step, the worst case model in such a large set is strongly against the controller which implies that, in the context of optimal portfolio, the money should only be invested in the banking account.
Dynamic reduction of uncertainty is thereby an apparent advantage maintained by AR over SR.
In practice, static robust control should only be used when there is sufficient historical data.
One still needs be cautious of potential estimation error as, for uncertainty set with small size, the SR optimal control will heavily depend on the initial guess of the unknown distribution.
Due to the lack of dynamic learning, SR optimal control in such case will be biased if the initial guess has large distance to the true model.
On the contrary, learning is incorporated in adaptive robust and thus the corresponding control will be almost optimal for time steps close to $T$, and this feature will be carried out to earlier time steps following the dynamic programming principle.

%%%%%%%%%%%%%%%%%%%%%%%%%%%%%%%%%%%%%%%%%%%%%%%%%%%%%%
\bibliographystyle{alpha}
\bibliography{nonparametric}

\end{document}